\documentclass{amsart}

\usepackage{amsmath, amsthm, amssymb,mathtools}
\usepackage{mathrsfs}
\usepackage[all]{xy}
\usepackage{paralist}
\usepackage[vcentermath]{youngtab}
\usepackage{tikz-cd}

\theoremstyle{definition}

\makeatletter
\newtheorem*{rep@theorem}{\rep@title}
\newcommand{\newreptheorem}[2]{%
\newenvironment{rep#1}[1]{%
 \def\rep@title{#2~\ref{##1}}%
 \begin{rep@theorem}}%
 {\end{rep@theorem}}}
 
\makeatother

\newtheorem{theorem}{Theorem}

\newtheorem{lemma}[theorem]{Lemma}
\newtheorem{proposition}[theorem]{Proposition}

\newtheorem{question}[theorem]{Question}

\newreptheorem{theorem}{Theorem}
\newreptheorem{proposition}{Proposition}

\theoremstyle{remark}
\newtheorem{remark}[theorem]{Remark}

\begin{document}

\title[Lefschetz classes on projective varieties]{Lefschetz classes on projective varieties}

\author{June Huh and Botong Wang}

\address{Institute for Advanced Study, Fuld Hall, 1 Einstein Drive, Princeton, NJ, USA.}
\email{huh@princeton.edu}


\address{University of Wisconsin-Madison, Van Vleck Hall, 480 Lincoln Drive, Madison, WI, USA.}
\email{bwang274@wisc.edu}

\begin{abstract}
The Lefschetz algebra $L^*(X)$ of a smooth complex projective variety $X$ is  the subalgebra of the cohomology algebra of $X$ generated by divisor classes. We construct smooth complex projective varieties  whose Lefschetz algebras do not satisfy analogues of the hard Lefschetz theorem and Poincar\'e duality.
 \end{abstract}

\maketitle

\section{Introduction}

Let $X$ be a $d$-dimensional smooth complex projective variety, and let $\text{Alg}^*(X)$ be the commutative graded $\mathbb{Q}$-algebra of algebraic cycles on $X$ modulo homological equivalence
\[
\text{Alg}^*(X)=\bigoplus_{k=0}^d \text{Alg}^k(X) \subseteq H^{2*}(X,\mathbb{Q}).
\]
A hyperplane section of $X \subseteq \mathbb{P}^n$ defines a cohomology class $\omega \in \text{Alg}^1(X)$.
Grothendieck's standard conjectures  predict that $\text{Alg}^*(X)$ satisfies analogues of  the hard Lefschetz theorem and Poincar\'e duality:
\begin{enumerate}[(1)]\itemsep 3pt
\item[(HL)] For every nonnegative integer $k \le \frac{d}{2}$, the linear map
\[
\text{Alg}^k(X)\longrightarrow \text{Alg}^{d-k}(X), \qquad x \longmapsto \omega^{d-2k} \ x
\]
is an isomorphism.
\item[(PD)] For every nonnegative integer $k \le \frac{d}{2}$, the bilinear map
\[
\text{Alg}^k(X)\times \text{Alg}^{d-k}(X)\longrightarrow \text{Alg}^{d}(X)\simeq \mathbb{Q}, \qquad (x_1,x_2) \longmapsto x_1x_2
\]
is nondegenerate.
\end{enumerate}
The properties (HL) and (PD) for $\text{Alg}^*(X)$ are implied by  the Hodge conjecture for $X$.

The \emph{Lefschetz algebra} of $X$ is  the graded  $\mathbb{Q}$-subalgebra $L^*(X)$ of $\text{Alg}^*(X)$ generated by divisor classes. When $X$ is singular, we may define the Lefschetz algebra  to be the graded $\mathbb{Q}$-algebra in the intersection cohomology
generated by the Chern classes of line bundles 
\[
L^*(X) \subseteq IH^{2*}(X,\mathbb{Q}).
\]
 An application of Lefschetz algebras to the ``top-heavy'' conjecture in enumerative combinatorics was  given in \cite{Huh-Wang}.

It was asked whether there are smooth projective varieties whose Lefschetz algebras do not 
satisfy analogues of the hard Lefschetz theorem and Poincar\'e duality \cite{Kaveh}:
\begin{enumerate}[(1)]\itemsep 3pt
\item[(HL)] For every nonnegative integer $k \le \frac{d}{2}$, the linear map
\[
L^k(X)\longrightarrow L^{d-k}(X), \qquad x \longmapsto \omega^{d-2k} \ x
\]
is an isomorphism.
\item[(PD)] For every nonnegative integer $k \le \frac{d}{2}$, the bilinear map
\[
L^k(X)\times L^{d-k}(X)\longrightarrow L^{d}(X)\simeq \mathbb{Q}, \qquad (x_1,x_2) \longmapsto x_1x_2
\]
is nondegenerate.
\end{enumerate}
We show in Proposition \ref{PropositionEquivalence} that (HL) and (PD) for $L^*(X)$ are equivalent to each other, and to the numerical condition that
\[
\text{dim}\ L^k(X)=\text{dim}\ L^{d-k}(X) \ \ \text{for all $k$}.
\]

Many familiar smooth projective varieties satisfy (HL) and (PD) for $L^*(X)$. In Section \ref{symmetric}, we show that this is the case for 
\begin{enumerate}[(1)]\itemsep 3pt
\item toric varieties, 
\item abelian varieties,
\item Grassmannians and full flag varieties,
\item wonderful compactifications of arrangement complements,
\item products of two or more varieties  listed above,
\item complete intersections of ample divisors in the varieties listed above, 
\item all smooth projective varieties with Picard number $1$, and
\item all smooth projective varieties of dimension at most $4$.
\end{enumerate}
In Section \ref{nonsymmetric}, we construct three varieties whose Lefschetz algebras do not satisfy (HL) and (PD).

\begin{theorem}
There is a $d$-dimensional smooth complex projective variety $X$ with the property
\[
\text{dim}\ L^2(X) \neq \text{dim}\ L^{d-2}(X).
\]
\end{theorem}

The first example has dimension $5$ and Picard number $2$,
 and the second example has dimension $6$ and Picard number $3$.
The third example is a partial flag variety;  it shows that the assumption made in \cite[Theorem 4.1]{Kaveh} is not redundant.

\section{Lefschetz algebras with Poincar\'e duality}\label{symmetric}

Let $X \subseteq \mathbb{P}^n$ be a $d$-dimensional smooth complex projective variety, and let $\omega\in L^1(X)$ be the cohomology class of a hyperplane section. 

\begin{proposition}\label{PropositionEquivalence}
The following statements are equivalent.
\begin{enumerate}[(1)]\itemsep 3pt
 \item The Lefschetz algebra $L^*(X)$ satisfies the hard Lefschetz theorem, that is, the linear map
 \[
L^k(X)\longrightarrow L^{d-k}(X), \qquad x \longmapsto \omega^{d-2k} \ x
\]
is an isomorphism for every nonnegative integer $k \le \frac{d}{2}$.
\item The Lefschetz algebra $L^*(X)$ satisfies Poincar\'e duality, that is, the bilinear map
\[
L^k(X)\times L^{d-k}(X)\longrightarrow L^{d}(X)\simeq \mathbb{Q}, \qquad (x_1,x_2) \longmapsto x_1x_2
\]
is nondegenerate for every nonnegative integer $k \le \frac{d}{2}$.
\item The Lefschetz algebra $L^*(X)$ has symmetric dimensions, that is, the equality
\[
\text{dim}\ L^k(X)=\text{dim}\ L^{d-k}(X) 
\]
holds for every nonnegative integer $k \le \frac{d}{2}$.
\end{enumerate}
\end{proposition}

\begin{proof}
Clearly, (1) implies (3), and (2) implies (3). 
The hard Lefschetz theorem for $H^*(X,\mathbb{Q})$ shows that (3) implies (1).
We prove that (3) implies (2).

Suppose (3), or equivalently, (1). This implies that, for every nonnegative integer $k\le \frac{d}{2}$,
\[
L^k(X)= \bigoplus_{i=0}^k \ \omega^{k-i} PL^{i}(X),
\ \ \text{where} \ \
PL^{i}(X)=\text{ker}\Big(\omega^{d-2i+1}:
L^i(X) \longrightarrow L^{d-i+1}(X)
\Big).
\]
In other words, every primitive component appearing in the Lefschetz decomposition of an element $x \in L^k(X)$  is an element of $L^*(X)$.
Let us write
\[
x=\sum_{i=0}^k \omega^{k-i} x_i, \quad x_i \in L^i(X).
\]
If $x$ is nonzero, then some summand $\omega^{k-j}x_j$ is nonzero.
The Hodge-Riemann relation for the primitive subspace of $H^{j,j}(X)$ shows that
\[
(-1)^j\int_X  \omega^{d-j-k}  x_jx=(-1)^j\int_X  \omega^{d-2j}  x_j^2>0.
\]
Thus the product of $x$ with $\omega^{d-j-k} x_j$ is nonzero, and hence $L^*(X)$ satisfies Poincar\'e duality.
\end{proof}

Many familiar smooth projective varieties  have Lefschetz algebras satisfying Poincar\'e duality.
The most obvious examples are the varieties with
$L^*(X)=H^{2*}(X,\mathbb{Q})$,
such as smooth projective toric varieties,  complete flag varieties, wonderful compactifications of hyperplane arrangement complements, etc.
We collect more examples in the remainder of this section.

\begin{lemma}\label{LemmaSmallk}
For every nonnegative integer $k \le \frac{d}{2}$, the linear map
 \[
L^k(X)\longrightarrow L^{d-k}(X), \qquad x \longmapsto \omega^{d-2k} \ x
\]
is injective. For $k=0$ and $k=1$, the map is bijective.
\end{lemma}

\begin{proof}
We prove the assertion for $k=1$.
By the Lefschetz $(1,1)$ theorem, we have
\[
L^1(X)= H^2(X,\mathbb{Q}) \cap H^{1,1}(X). 
\]
The hard Lefschetz theorem for $H^*(X,\mathbb{Q})$ implies that the left-hand side is isomorphic to
\[
H^{2d-2}(X,\mathbb{Q}) \cap H^{d-1,d-1}(X),  
\]
which contains $L^{d-1}(X)$ as a subspace. 
This forces $\text{dim}\ L^1(X)=\text{dim}\ L^{d-1}(X)$.
\end{proof}

\begin{proposition}\label{goodones}
Suppose any one of the following conditions:
\begin{enumerate}[(1)]\itemsep 3pt
\item\label{c1} $X$ is an abelian variety.
\item\label{c2} $X$ has Picard number $1$.
\item\label{c3} $X$ has dimension at most $4$.
\end{enumerate}
Then $L^*(X)$ satisfies Poincar\'e duality. 
\end{proposition}
\begin{proof}
(\ref{c1}) is proved by Milne \cite[Proposition 5.2]{Milne}. (\ref{c2}) and
 (\ref{c3})  follow from Proposition \ref{PropositionEquivalence} and Lemma \ref{LemmaSmallk}.
\end{proof}

We may construct Lefschetz algebras with Poincar\'e duality by taking hyperplane sections. 
Let $\iota$ be the inclusion of a smooth ample hypersurface $D \subseteq X$ with cohomology class $\omega \in L^1(X)$.

\begin{proposition}\label{hyperplane}
 If $L^*(X)$ satisfies Poincar\'e duality, then $L^*(D)$ satisfies Poincar\'e duality. 
\end{proposition}
\begin{proof}
The Lefschetz hyperplane theorem for $D \subseteq X$ and Poincar\'e duality for $D$ show that the pullback  $\iota^*$ in cohomology induces a commutative diagram
 \[
  \raisebox{-0.5\height}{\includegraphics{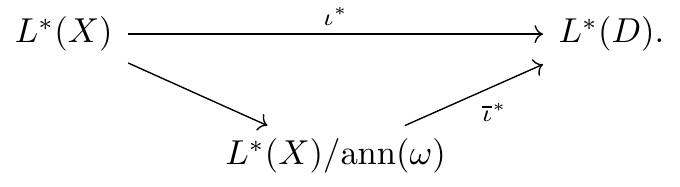}}
 \]
If $L^*(X)$ satisfies Poincar\'e duality, then $L^*(X)/\text{ann}(\omega)$ satisfies Poincar\'e duality:
\[
L^k(X)/\text{ann}(\omega) \times L^{d-k-1}(X)/\text{ann}(\omega) \longrightarrow L^{d-1}(X)/\text{ann}(\omega) \simeq L^d(X) \simeq \mathbb{Q} \ \ \text{is nondegenerate.}
\]
When $d <4$,  the conclusion follows from
Proposition \ref{goodones} applied to $D$.
Assuming $d \ge 4$, we show that the induced map $\overline{\iota}^*$ is an isomorphism. 

When $d \ge 4$,
the Lefschetz hyperplane theorem for Picard groups applies to the inclusion $D \subseteq X$,
and therefore $\overline{\iota}^*$ is surjective \cite[Chapter 3]{Lazarsfeld}. 
To conclude, we deduce from the hard Lefschetz theorem for $H^*(X,\mathbb{Q})$  that
\[
L^k(X) \simeq L^k(X)/\text{ann}(\omega) \simeq L^k(D) \ \ \text{for every nonnegative integer $k < d/2$}.
\] 
Lemma \ref{LemmaSmallk} applied to $D$ shows that $\overline{\iota}^*$ is an injective in the remaining degrees $k \ge \frac{d}{2}$.
\end{proof}

We may construct Lefschetz algebras with Poincar\'e duality by taking products. 
Let $X_1$ and $X_2$ be smooth  projective varieties, and suppose that $H^1(X_1,\mathbb{Q})=0$.

\begin{proposition}\label{product}
There is an isomorphism between graded algebras
\[
L^*(X_1 \times X_2) \simeq L^*(X_1) \otimes_\mathbb{Q} L^*(X_2).
\]
Thus $L^*(X_1 \times X_2)$ satisfies Poincar\'e duality if $L^*(X_1)$ and $L^*(X_2)$ satisfy Poincar\'e duality.
 \end{proposition}
 
\begin{proof}
By the K\"unneth formula, there is an isomorphism of Hodge structures
$$
H^2(X_1\times X_2, \mathbb{Q})\simeq H^2(X_1, \mathbb{Q}) \oplus H^2(X_2, \mathbb{Q}).
$$
The above restricts to an isomorphism between subspaces 
\[
L^1(X_1\times X_2)\simeq L^1(X_1)\oplus L^1(X_2),
\]
which induces an isomorphism of graded $\mathbb{Q}$-algebras 
$L^*(X_1\times X_2)\simeq L^*(X)\otimes_\mathbb{Q} L^*(X_2)$.
\end{proof}

Proposition \ref{goodones}, Proposition \ref{hyperplane},  and Proposition \ref{product} justify that the eight classes of smooth projective varieties  listed in the introduction have  Lefschetz algebras with Poincar\'e duality.

\section{Lefschetz algebras without Poincar\'e duality}\label{nonsymmetric}
We construct three  smooth  projective varieties whose Lefschetz algebras do not satisfy Poincar\'e duality. Before giving the construction, we recall standard description of the cohomology ring of a blowup. 

Let $Z$ be a codimension $r$ smooth subvariety of a $d$-dimensional smooth projective variety $Y$. We write $\pi:X \to Y$ for  the blowup of $Y$ along $Z$,  and $\iota:Z \to Y$ for the inclusion of $Z$ in $Y$:
\[
  \raisebox{-0.5\height}{\includegraphics{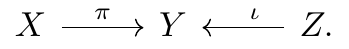}}
\]
The cohomology ring of the blowup $X$ can be described as follows \cite[Chapter 4]{GH}.

\begin{proposition}\label{blowup}
There is a decomposition of graded vector spaces
\begin{equation}\label{iso1}
H^{*}(X,\mathbb{Q})\simeq \pi^*H^{*}(Y,\mathbb{Q})\oplus \Bigg( \sum_{i=1}^{r-1}H^{*-2i}(Z,\mathbb{Q})  \otimes \mathbb{Q} \hspace{0.3mm}e^i\Bigg),
\end{equation}
where $e$ is the cohomology class of the exceptional divisor in $X$. 
The cup product satisfies
\[
\pi^*y \cup (z \otimes e^i)=(\iota^*y \cup z)\otimes e^i \ \ \text{for cohomology classes $y$ of $Y$ and $z$ of $Z$},
\]
and, writing $N_{Z/Y}$ for the normal bundle of the embedding $\iota:Z \to Y$, we have
\[
(-1)^re^r=\pi^*\iota_*(1)-\sum_{i=1}^{r-1} c_{r-i}(N_{Z/Y}) \otimes (-e)^{i}.
\]
\end{proposition}

We abuse notation and suppress the symbols $\iota^*$ and $\pi^*$ in computations below.

\subsection{A $5$-dimensional example}
According to Proposition \ref{goodones}, if $X$ is a smooth projective variety whose Lefschetz algebra does not satisfy Poincar\'e duality, then the dimension of $X$ is at least $5$ and the Picard number of $X$ is at least $2$. We construct such an example of dimension $5$ and Picard number $2$. 

Let $C \subseteq \mathbb{P}^2$ be a smooth plane cubic curve,  let $Y$ be the projective space $\mathbb{P}^5$, and let $Z$ be the product $C \times \mathbb{P}^1$.
We write $\iota$ for the composition of inclusions
\[
  \raisebox{-0.5\height}{\includegraphics{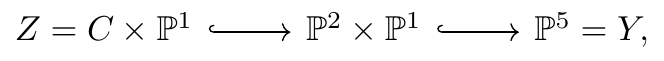}}
\]
where  the second map is the Segre embedding.
Let $X$ be the blowup of $Y$ along $Z$.

\begin{proposition}\label{ex1}
$\text{dim}\ L^2(X)=3$ and $\text{dim}\ L^3(X)=4$.
\end{proposition}

\begin{proof}
We have $H^2(Z,\mathbb{Q})=\mathbb{Q} \hspace{0.5mm} a  \oplus \mathbb{Q} \hspace{0.5mm} b$,
where $a$ and $b$ are cohomology classes of $C\times \text{point}$ and $\text{point} \times \mathbb{P}^1$ respectively.
Writing $c$ for the cohomology class of a hyperplane in $Y$,
we have
\begin{align*}
H^0(X,\mathbb{Q})&= \mathbb{Q}\hspace{0.5mm} 1, \\
H^2(X,\mathbb{Q})&=  \mathbb{Q}\hspace{0.5mm} c^1 \oplus \big( \mathbb{Q}\hspace{0.5mm} 1 \big)e,\\
H^4(X,\mathbb{Q})&= \mathbb{Q}\hspace{0.5mm} c^2  \oplus\big( \mathbb{Q}\hspace{0.5mm} a  \oplus \mathbb{Q}\hspace{0.5mm} b \big)e \oplus \big(\mathbb{Q}\hspace{0.5mm} 1 \big) e^2,\\
H^6(X,\mathbb{Q})&= \mathbb{Q}\hspace{0.5mm} c^3 \oplus\big( \mathbb{Q}\hspace{0.5mm} ab \big) e \oplus \big( \mathbb{Q}\hspace{0.5mm} a  \oplus \mathbb{Q}\hspace{0.5mm} b \big) e^2,\\
H^8(X,\mathbb{Q})&= \mathbb{Q}\hspace{0.5mm} c^4 \oplus \big(\mathbb{Q}\hspace{0.5mm}ab \big)e^2,\\
H^{10}(X,\mathbb{Q})&=\mathbb{Q}\hspace{0.5mm} c^5.
\end{align*}
where $e$ is the cohomology class of the exceptional divisor in $X$.

The algebra $L^*(X)$ is generated by $c$ and $e$. The restriction  of $c$ to $Z$ is $a+3b$, and hence
\[
L^2(X)=\mathbb{Q}\hspace{0.5mm} c^2  \oplus \mathbb{Q} \hspace{0.5mm}ce \oplus \mathbb{Q}\hspace{0.5mm}  e^2=\mathbb{Q}\hspace{0.5mm} c^2  \oplus \mathbb{Q}( a  + 3b)e \oplus \mathbb{Q}\hspace{0.5mm}  e^2.
\]
This proves the first assertion.

We next show 
$L^3(X)=H^6(X,\mathbb{Q})$. It is enough to check that $e^3$ is not  in the subspace
\[
V=\mathbb{Q} c^3 \oplus \mathbb{Q} c^2e \oplus\mathbb{Q}ce^2=\mathbb{Q}\hspace{0.5mm} c^3 \oplus \mathbb{Q}(ab) e \oplus  \mathbb{Q}( a  + 3b) e^2 \subseteq H^6(X,\mathbb{Q}).
\]
According to Proposition \ref{blowup}, the following relation holds in the cohomology of $X$:
\[
e^3=-6c^3-c_2(N_{Z/Y})\hspace{0.5mm} e+c_1(N_{Z/Y})\hspace{0.5mm} e^2=-6c^3-c_2(N_{Z/Y})\hspace{0.5mm} e+c_1(T_Y)\hspace{0.5mm} e^2-c_1(T_Z)\hspace{0.5mm} e^2.
\]
Since $c_1(T_Y)\hspace{0.5mm} e^2$ is a multiple of $ce^2$ and   $c_2(N_{Z/Y})\hspace{0.5mm} e$ is a  multiple of $abe$, we have
\[
e^3=-c_1(T_Z)\hspace{0.5mm} e^2 \ \ \text{mod} \ \ V
\]
The tangent bundle of the elliptic curve $C$ is trivial, and therefore $c_1(T_Z)\hspace{0.5mm} e^2$ must be a multiple of $ae^2$.
It follows that $e^3$ is not contained in $V$. 
This proves the second assertion.
\end{proof}

\subsection{A $6$-dimensional example}

A simpler example can be found in dimension $6$.
Let $Y=\mathbb{P}^3\times \mathbb{P}^3$, and let $Z=\mathbb{P}^1\times \mathbb{P}^1\times \mathbb{P}^1$. We write $\iota$ for the composition of inclusions
\[
  \raisebox{-0.5\height}{\includegraphics{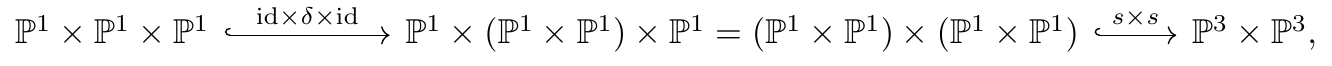}}
\]
where $\text{id}$ is the identity map, $\delta$ is the diagonal embedding, and $s$ is the Segre embedding.
Let $X$ be the blowup of $Y$ along $Z$.

\begin{proposition}
$\text{dim}\ L^2(X)=6$ and $\text{dim}\ L^4(X)=7$.
\end{proposition}

\begin{proof}
Write
$y_1$, $y_2$ for the cohomology classes 
\[
\text{cl}(\mathbb{P}^2\times\mathbb{P}^3),\ 
\text{cl}(\mathbb{P}^3\times\mathbb{P}^2) \in H^2(\mathbb{P}^3 \times\mathbb{P}^3,\mathbb{Q}),
\]
and write $z_1$, $z_2$, $z_3$ for the cohomology classes 
\[
\text{cl}(\mathbb{P}^0\times\mathbb{P}^1 \times \mathbb{P}^1),\ 
\text{cl}(\mathbb{P}^1\times\mathbb{P}^0 \times \mathbb{P}^1), \ 
\text{cl}(\mathbb{P}^1\times\mathbb{P}^1 \times \mathbb{P}^0) \in H^2(\mathbb{P}^1\times\mathbb{P}^1 \times \mathbb{P}^1,\mathbb{Q}).
\]
Note  that $\iota^*y_1=z_1+z_2$ and
$\iota^*y_2=z_2+z_3$.
According to Proposition \ref{blowup}, 
\begin{align*}
H^4(X,\mathbb{Q})&=\mathbb{Q} \hspace{0.5mm} y_1^2 \oplus \mathbb{Q} \hspace{0.5mm} y_1y_2 \oplus \mathbb{Q} \hspace{0.5mm} y_2^2
 \oplus\big( \mathbb{Q} \hspace{0.5mm} z_1 \oplus \mathbb{Q} \hspace{0.5mm} z_2 \oplus \mathbb{Q} \hspace{0.5mm} z_3 \big)e
 \oplus \big(\mathbb{Q}\hspace{0.5mm}1\big) e^2 \ \text{and}\\[3pt]
H^8(X,\mathbb{Q})&= \mathbb{Q} \hspace{0.5mm} y_1^3y_2 \oplus \mathbb{Q} \hspace{0.5mm} y_1^2 y_2^2 \oplus \mathbb{Q} \hspace{0.5mm} y_1y_2^3 
  \oplus \big(\mathbb{Q} \hspace{0.5mm} z_1z_2z_3\big)e \oplus \big(\mathbb{Q} \hspace{0.5mm} z_2z_3 \oplus \mathbb{Q} \hspace{0.5mm} z_1z_3 \oplus \mathbb{Q} \hspace{0.5mm} z_1z_2 \big)e^2,
\end{align*}
where $e$ is the cohomology of the exceptional divisor in $X$.

The vector space $L^2(X)$ is spanned by  the cohomology classes $y_1^2$, $y_1y_2$, $y_2^2$, $e^2$, 
\[
y_1e= z_1e+z_2e,\ \ \text{and} \ \  y_2e=z_2e+z_3e.
\]
From the above description of $H^4(X,\mathbb{Q})$, we see that the six elements are linearly independent. This proves the first assertion.

We next check $L^4(X)=H^8(X,\mathbb{Q})$.
Note that $L^4(X)$ contains $y_1^3y_2$, $y_1^2y_2^2$, $y_1y_2^3$, and
\[
y_1^2y_2e=2z_1z_2z_3e,\quad
y_1^2e^2=2z_1z_2e^2, \quad
y_2^2e^2=2z_2z_3e^2,\quad
y_1y_2e^2=(z_1z_2+z_1z_3+z_2z_3)e^2.
\]
From the above description of $H^8(X,\mathbb{Q})$, we see that the seven elements span $H^8(X,\mathbb{Q})$.
This proves the second assertion.
\end{proof}

\subsection{An $8$-dimensional example}

Let $X$ be the $8$-dimensional partial flag variety
\[
X=\Bigg\{0 \subseteq V_2 \subseteq V_3 \subseteq \mathbb{C}^5 \mid \text{dim}\ V_2=2, \  \text{dim}\ V_3=3\Bigg\}.
\]
We show that the Lefschetz algebra of $X$ does not satisfy Poincar\'e duality.
This is in contrast with the case of Grassmannians and full flag varieties.

\begin{proposition}
$\text{dim}\ L^2(X)=3$ and $\text{dim}\ L^6(X)=4$. 
\end{proposition}

\begin{proof}
We use standard facts and notations in Schubert calculus \cite{Eisenbud-Harris}.
Let $Y$ be the Grassmannian variety parametrizing $2$-dimensional subspaces of $\mathbb{C}^5$.
The Schubert classes form a basis of the cohomology of $Y$:
\begin{align*}
H^*(Y,\mathbb{Q})=\mathbb{Q}\ \yng(3,3) \oplus\mathbb{Q}\ \yng(3,2) \oplus\mathbb{Q}\ \yng(3,1) \oplus\mathbb{Q}\ \yng(2,2) \oplus\mathbb{Q}\ \yng(3)\oplus\mathbb{Q}\ \yng(2,1)  \oplus\mathbb{Q}\ \yng(1,1) \oplus\mathbb{Q}\ \yng(2) \oplus\mathbb{Q}\ \yng(1) \oplus \mathbb{Q}\ 1.
\end{align*}
We write $\mathcal{Q}$ for the universal quotient bundle on $Y$. 
Since $X$  is the projectivization $\mathbb{P}\mathcal{Q}$,
\begin{align*}
H^4(X,\mathbb{Q})&= \mathbb{Q}\ \yng(1,1) \oplus \mathbb{Q} \ \yng(2) \oplus \mathbb{Q}\ \yng(1)\ \zeta \oplus \mathbb{Q}\ \zeta^2 \ \text{and}\\[3pt]
H^{12}(X,\mathbb{Q})&= \mathbb{Q}\ \yng(3,3) \oplus \mathbb{Q}\  \yng(3,2)\ \zeta \oplus \mathbb{Q}\  \yng(3,1)\ \zeta^2 \oplus \mathbb{Q}\  \yng(2,2)\ \zeta^2,
\end{align*}
where $\zeta$ is the first Chern class of the line bundle  $\mathcal{O}_{\mathbb{P}\mathcal{Q}}(1)$.

The Lefschetz algebra of $X$ is generated by $\yng(1)$ and $\zeta$, and therefore
\[
L^2(X)= \mathbb{Q}\ \yng(1)^2 \oplus \mathbb{Q} \ \yng(1) \ \zeta \oplus  \mathbb{Q} \ \zeta^2
= \mathbb{Q}\ \Big( \yng(2)+\yng(1,1)\ \Big) \oplus \mathbb{Q} \ \yng(1) \ \zeta \oplus  \mathbb{Q} \ \zeta^2.
\]
This proves the first assertion.

We now show that $L^6(X)=H^{12}(X,\mathbb{Q})$.
For this we use four elements
\[
\yng(1)^6, \ \yng(1)^5\ \zeta, \ \yng(1)^4\ \zeta^2,\  \yng(1)^2\ \zeta^4 \in L^6(X).
\]
It is enough to prove that the last element $\yng(1)^2\ \zeta^4$ is not contained in the subspace
\[
V=\mathbb{Q}\ \yng(1)^6 \oplus \mathbb{Q} \ \yng(1)^5\ \zeta \oplus \mathbb{Q}\  \yng(1)^4\ \zeta^2
=
\mathbb{Q}\ \yng(3,3) \oplus \mathbb{Q} \ \yng(3,2)\ \zeta \oplus \mathbb{Q} \Big( 2\ \yng(2,2)+3\ \yng(3,1)\ \Big)  \zeta^2.
\]
According to \cite[Chapter 5]{Eisenbud-Harris}, the Chern classes of $\mathcal{Q}$ are 
\[
c_0(\mathcal{Q})=1, \ \
c_1(\mathcal{Q})=\yng(1), \ \ 
c_2(\mathcal{Q})=\yng(2),  \ \ 
c_3(\mathcal{Q})=\yng(3) \ \in \  H^*(Y,\mathbb{Q}).
\]
In other words,
$\zeta^3+ \yng(1)\ \zeta^2+\yng(2)\ \zeta +\yng(3)=0$ in the cohomology of $X=\mathbb{P}\mathcal{Q}$.
It follows that
\begin{align*}
\zeta^4&= -\yng(3)\ \zeta-\yng(2)\ \zeta^2- \yng(1)\ \zeta^3\\
&=  -\yng(3)\ \zeta-\yng(2)\ \zeta^2-\yng(1)\ \Big( -\yng(3)-\yng(2)\ \zeta-\yng(1)\ \zeta^2\Big)   =  \yng(3,1)+\yng(2,1)\ \zeta+ \yng(1,1)\ \zeta^2, 
\end{align*}
and therefore
\[
\yng(1)^2\ \zeta^4=\yng(1)^2\ \Big(\   \yng(3,1) +\yng(2,1)\ \zeta+\yng(1,1)\ \zeta^2  \ \Big)=\yng(3,3) + 2\ \yng(3,2)\ \zeta+\Big( \ \yng(2,2)+\yng(3,1)\ \Big)  \zeta^2 \notin V.
\]
This proves the second assertion.
\end{proof}

\subsection*{Acknowledgements}
June Huh thanks H\'el\`ene Esnault for valuable conversations.
This research was conducted when the authors were visiting Korea Institute for Advanced Study. We thank KIAS for excellent working conditions.
June Huh was supported by a Clay Research Fellowship and NSF Grant DMS-1128155.

\end{document}